\documentclass[a4paper, 12pt]{article}
\usepackage[utf8]{inputenc}
\usepackage[top=1in, bottom=1.26in, left=1in, right=1in]{geometry}
\usepackage{amsmath, amssymb, amsthm}
\usepackage{mathrsfs}
\usepackage{enumerate}
\usepackage{hyperref}
\usepackage{color} 
\usepackage{float}

\usepackage{tikz}
\usetikzlibrary{arrows,calc,cd,chains,decorations.markings,decorations.pathmorphing,
	graphs,graphs.standard,patterns,positioning,quotes,scopes,shapes,shapes.geometric}

\newtheorem{definition}{Definition}[section]
\newtheorem{proposition}{Proposition}[section]
\newtheorem{theorem}[proposition]{Theorem}
\newtheorem{lemma}[proposition]{Lemma}
\newtheorem{remark}[proposition]{Remark}

\theoremstyle{definition}
\newtheorem{example}{Example}[section]

\numberwithin{equation}{section}

\DeclareMathOperator{\diam}{diam}

\title{{\bf Equivalence of Variants of Shadowing of Free Semigroup Actions}}   
\author{Pramod Kumar Das, Priyabrata Bag\\
Department of Mathematics, GITAM School of Science,\\
Gandhi Institute of Technology and Management,\\
Doddaballapur Taluk, Bengaluru, Karnataka 561203, India\\
pramod.math.ju@gmail.com; priyabrata.bag@gmail.com} 
\date{}

\begin{document}
\maketitle

\begin{abstract}
\noindent We prove that for finitely generated free semigroup actions the average shadowing property, the weak asymptotic average shadowing property, the mean ergodic shadowing property, the almost asymptotic average shadowing property, the asymptotic average shadowing property and the $M_{\alpha}$-shadowing property for every $\alpha\in (0,1)$, are equivalent. This gives an affirmative answer to an open question asked in Question 10.3 [M. Kulczycki, D. Kwietniak, P. Oprocha, On almost specification and average shadowing properties, Fundamenta Mathematicae, 224 (2014)].

\mbox{}\\

\noindent{\bf Keywords:} Shadowing, Average Shadowing, Ergodic Shadowing, Sub-shadowing.\\     
\mbox{}\\
\noindent{\bf Mathematics Subject Classification (2020): 37B65, 37B05}   
\end{abstract} 

\section{Introduction} 

\hspace{0.6 cm} A significant portion of the theory of topological dynamics (see \cite{AH}) deals with exploring the long-term behaviour of the systems controlled by a fixed rule. Such systems are known as autonomous dynamical systems. However, most dynamical systems describing real-world phenomena like population growth of species, spread of infectious diseases, pattern of heartbeats etc., are affected by finitely many different sets of rules over time. For this reason, it is important to understand the dynamics of finitely generated semigroup actions. Such studies also provide better insights into the behaviour of autonomous dynamical systems as explained by Theorem \ref{Autonomous}.  
\medskip

The shadowing property and its variants play a crucial role in the theory of topological dynamics (see \cite{AH}). The most extensively studied variant is known as the average shadowing property introduced in \cite{B}. Its role in understanding the chaotic behaviour of dynamical systems gave impetus to the introduction and investigation of several other variants of shadowing with average error in tracing (see \cite{WOC}), namely, the asymptotic average shadowing property (see \cite{G}), the weak asymptotic average shadowing property (see \cite{WOC}), the almost asymptotic average shadowing property (see \cite{GD}), the mean ergodic shadowing property (see \cite{DD}) and so on. It is the notion of sub-shadowing ($\underline{d}$-shadowing and $\overline{d}$-shadowing) (see \cite{DH}) which was introduced using the concept of density of subsets of natural numbers. Some research papers explored interrelations among some of these variants for autonomous dynamical systems. For instance, in \cite{WOC}, authors proved that the average shadowing property, the weak asymptotic average shadowing property, and the $M_{\alpha}$-shadowing property for every $\alpha\in [0,1)$, are equivalent. In \cite{WOC}, it was also proved that the asymptotic average shadowing property implies the average shadowing property. In \cite{DD}, authors proved that the mean ergodic shadowing property implies the $M^{\alpha}$-shadowing property for every $\alpha\in (0,1)$.    
\medskip

In recent years, researchers also studied (see \cite{B2015}, \cite{S}, \cite{ST}, \cite{N}, \cite{ZWFL}) the shadowing property and some of its variants for systems controlled by more than one rule of evolution. In \cite{S}, it was proved that any finitely generated semigroup with the ergodic shadowing property has the shadowing property. In \cite{ST}, authors proved that the $\underline{d}$-shadowing property, the average shadowing property, the weak specification property and the pseudo-orbital specification property are equivalent for a finitely generated semigroup with the shadowing property. In \cite{ZWFL}, it was proved that the asymptotic average shadowing property implies the average shadowing property for such systems.       
\medskip

Our purpose here is to prove the following result which shows that many of the variants of shadowing are equivalent for a free semigroup action of $\Gamma=\Gamma(f_0,f_1, \ldots, f_m)$ generated by the family of continuous self-maps $\mathcal{F}=\lbrace f_i: X\rightarrow X\rbrace_{i=0}^m$, where $f_0=id$ is the identity map on the compact metric space $X$.   

\begin{theorem} \label{MT} 
For a semigroup $\Gamma=\Gamma(f_0,f_1, \ldots, f_m)$ and $w\in \lbrace 1, \ldots, m\rbrace^{\mathbb{N}}$, the following are equivalent:   
\begin{itemize}
\item [(a)] $w$-average shadowing property,
\item [(b)] $w$-weak asymptotic average shadowing property,
\item [(c)] $w$-mean ergodic shadowing property,
\item [(d)] $w$-almost asymptotic average shadowing property, 
\item [(e)] $w$-$M_{\alpha}$-shadowing property for every $\alpha\in (0,1)$, 
\item [(f)] $w$-asymptotic average shadowing property.     
\end{itemize}
\end{theorem}

If $(X,f)$ is an autonomous dynamical system, its dynamics are the same as the dynamics of the free semigroup action of $\Gamma=\Gamma(f_0,f_1=f)$, if $w$ happens to be the constant sequence $\lbrace 1,1,1, \ldots \rbrace$.  Therefore, proving Theorem \ref{MT} also proves the following result:         

\begin{theorem}\label{Autonomous} 
Let $f: X\rightarrow X$ be a continuous map on a compact metric space $X$. Then the following are equivalent: \begin{itemize}
\item [(1)] The average shadowing property,
\item [(2)] The weak asymptotic average shadowing property,
\item [(3)] The mean ergodic shadowing property,
\item [(4)] The almost asymptotic average shadowing property,
\item [(5)] The $M_{\alpha}$-shadowing property for every $\alpha\in (0,1)$, 
\item [(6)] The asymptotic average shadowing property.  
\end{itemize} 
\end{theorem} 

This result answers fully the question ``How is the AASP related to the average shadowing?'' asked in \cite{NWS}. It also give an affirmative answer to the question ``Does the average shadowing property implies the asymptotic average shadowing property?" asked in Question 10.3 \cite{KKO}. Earlier in \cite{WOC}, it was proved that the asymptotic average shadowing property implies the average shadowing property. In \cite{CT}, the equivalence between the asymptotic average shadowing property and the asymptotic average shadowing property was established for systems with those phase spaces which are complete with respect to the dynamical Besicovitch pseudometric. Further, this result shows that the properties listed in Theorem 5.5 of \cite{WOC} are also equivalent to the mean ergodic shadowing property (see \cite{DD}), the almost asymptotic average shadowing property (see \cite{GD}) and the asymptotic average shadowing property (see \cite{G}). It is clear from our proof that the mean ergodic shadowing property played an important role in establishing the equivalence among six of the variants of (average) shadowing.
\medskip

One may refer to Section~\ref{def-pre} for the definitions and preliminaries required to understand the results and discussion in Section~\ref{main}.  

\section{Definitions and Preliminaries} \label{def-pre}

In this section, we first describe the finitely generated free semigroup actions. Let us denote by $\mathbb{N}$ and $\mathbb{Z}^+$ respectively, the set of all natural numbers $\lbrace 0, 1, 2, 3, \ldots\rbrace$ and the set of all positive integers $\lbrace 1, 2, 3, \ldots\rbrace$.   
\medskip

Let us assume that $\Gamma=\Gamma(f_0,f_1, \ldots, f_m)$ be the semigroup generated by the finite family \(\mathcal{F}_1= \{f_0, f_1, \ldots , f_m\}\), where $f_0:X\rightarrow X$ is the identity map and \(f_i
: X \to X\), \(i \in \{1, \ldots , m\}\) are continuous self-maps on the compact metric space \(X\). Let us write \(\Gamma =\bigcup_{n\in\mathbb{N}} \mathcal{F}_n\), where \(\mathcal{F}_0 = \lbrace id\rbrace\) and $\mathcal{F}_n =\{f_{i_n}\circ\cdots\circ f_{i_2}\circ f_{i_1}$ $\mid$ $f_{i_j}\in \mathcal{F}_1\}$.  
Indeed, \(\mathcal{F}_n\) consists of elements that are concatenations of \(n\) elements of \(\mathcal{F}_1\).  
\medskip



The symbolic dynamics is a way to display the elements of the semigroup \(\Gamma\) associated with $\mathcal{F}_1$. Let \(\Sigma^m=\lbrace 1, \ldots, m\rbrace^{\mathbb{N}}\), i.e., $\Sigma^m$ is the set of all infinite sequences in $\lbrace 1, \ldots, m\rbrace$. For a sequence \(w = w_0w_1w_2\cdots\in \Sigma^m\), take \(f^0_w=id\) and for \(n\geq 1\), \(f^n_w= f_{w_{n-1}}\circ\cdots\circ f_{w_0}\). Let \(\mathcal{A}^m\) be the set of finite words of symbols \(\{1, \ldots , m\}\), i.e., if \(w \in\mathcal{A}^m\), then \(w = w_0 \cdots w_{l-1}\), where \(w_i \in \{1, \ldots , m\}\) for all \(i = 0,\ldots , l - 1\). Also, for \(0 \leq i \leq l - 1\), we denote \(f^i_w := f_{w_{i-1}} \circ\cdots\circ f_{w_0}\). 
\medskip

We now define five different types of pseudo orbits which are traced by true orbits to introduce six different variants of the shadowing property listed in Theorem \ref{MT} and Theorem \ref{Autonomous}.     

\begin{definition} 
Let \(\Gamma=\Gamma(f_0,f_1, \ldots, f_m)\), \(\xi=(x_j)_{j\in\mathbb{N}}\) a sequence in \(X\) and \(w=(w_j)_{j\in\mathbb{N}}\in\Sigma^m\). Let $f:X\rightarrow X$ be a continuous map on a compact metric space $X$. Assume that \(\delta\) is a positive real number. Then
\begin{itemize}
\item [(a)] \(\xi\) is said to be a \((\delta, w)\)-pseudo orbit (resp. $\delta$-pseudo orbit) for \(\Gamma\) (resp. for $f$) if \(d(f_{w_{j}}(x_{j}),x_{j+1})<\delta\) (resp. \(d(f(x_{j}),x_{j+1})<\delta\)), for all $j\in\mathbb{N}$. 
\item [(b)] \(\xi\) is said to be a \((\delta,w)\)-ergodic pseudo orbit (resp. $\delta$-ergodic pseudo orbit) for \(\Gamma\) (resp. for $f$) if \(d(f_{w_{j}}(x_{j}),x_{j+1})<\delta\) (resp. \(d(f(x_{j}),x_{j+1})<\delta\)), for all $j\in\mathbb{N}$ 
except for a set of density zero.
\item [(c)] \(\xi\) is said to be a \((\delta,w)\)-average pseudo orbit (resp. $\delta$-average pseudo orbit) for \(\Gamma\) (resp. for $f$) if there is $N>0$ such that for all $n\geq N$ and $k\in \mathbb{N}$,

\(\frac{1}{n}\sum\limits_{j=0}^{n-1}d(f_{w_{j}}(x_{j+k}), x_{j+k+1})<\delta\) (resp. \(\frac{1}{n}\sum\limits_{j=0}^{n-1}d(f(x_{j+k}), x_{j+k+1})<\delta\)).
\item [(d)] \(\xi\) is said to be a \((\delta,w)\)-weak asymptotic average pseudo orbit (resp. $\delta$-weak asymptotic average pseudo orbit) for \(\Gamma\) (resp. for $f$) if
		
\(\limsup\limits_{n\rightarrow \infty}\frac{1}{n}\sum\limits_{j=0}^{n-1}d(f_{w_j}(x_{j}), x_{j+1})<\delta\) (resp. \(\limsup\limits_{n\rightarrow \infty}\frac{1}{n}\sum\limits_{j=0}^{n-1}d(f(x_{j}), x_{j+1})<\delta\)). 
\item [(e)] \(\xi\) is said to be a \(w\)-asymptotic average pseudo orbit (resp. asymptotic average pseudo orbit) for \(\Gamma\) (resp. for $f$) if   
		
\(\lim\limits_{n\rightarrow \infty}\frac{1}{n}\sum\limits_{j=0}^{n-1}d(f_{w_j}(x_{j}), x_{j+1})=0\) (resp. \(\lim\limits_{n\rightarrow \infty}\frac{1}{n}\sum\limits_{j=0}^{n-1}d(f(x_{j}), x_{j+1})=0\)).
\end{itemize} 
\end{definition}
For $A\subset\mathbb{N}$, the upper density and the lower density of $A$ are respectively defined by
$$\overline{d}(A)=\limsup\limits_{n\rightarrow\infty}\frac{1}{n}
\left|A\cap\{0,\,1,\ldots,\,n-1\}\right|$$ and $$\underline{d}(A)=\liminf\limits_{n\rightarrow\infty}\frac{1}{n}
\left|A\cap\{0,\,1,\ldots,\,n-1\}\right|$$ where $|B|$ denotes the number of elements present in the set $B$. If there exists a number $d(A)$ such that
$\overline{d}(A)=\underline{d}(A)=d(A)$, then we say that the set $A$ has density $d(A)$. Fix $\alpha\in (0,1)$
and denote by $M_\alpha$ the family consisting
of sets $A\subset\mathbb{N}$ with $\underline{d}(A)>\alpha$. 
\begin{definition}
Let \(\Gamma=\Gamma(f_0,f_1, \ldots, f_m)\) and \(w=(w_j)_{j\in\mathbb{N}}\in\Sigma^m\). Let $f:X\rightarrow X$ be a continuous map on a compact metric space $X$. Then,
\begin{enumerate}
\item [(1)] \(\Gamma\) (resp. $f$) is said to have the \(w\)-mean ergodic shadowing property (resp. mean ergodic shadowing property) if for every \(\epsilon>0\), there is \(\delta>0\) such that every \((\delta, w)\)-ergodic pseudo orbit (resp. $\delta$-ergodic pseudo orbit) $(x_j)_{j\in\mathbb{N}}$ is \((\epsilon, w)\)-shadowed on average (resp. $\epsilon$-shadowed on average) by a point \(z\in X\), i.e., 
		
$\limsup\limits_{n\rightarrow\infty}\frac{1}{n}\sum\limits_{j=0}^{n-1}d(f_{w}^{j}(z), x_{j})<\epsilon$ (resp. $\limsup\limits_{n\rightarrow\infty}\frac{1}{n}\sum\limits_{j=0}^{n-1}d(f^{j}(z), x_{j})<\epsilon$). 
		
\item [(2)]  \(\Gamma\) (resp. $f$) is said to have the \(w\)-average shadowing property (resp. average shadowing property) if for every $\epsilon>0$, there is $\delta>0$ such that every $(\delta,w)$-average pseudo orbit (resp. $\delta$-average pseudo orbit) $(x_j)_{j\in\mathbb{N}}$ is $(\epsilon,w)$-shadowed on average (resp. $\epsilon$-shadowed on average) by a point $z\in X$, i.e., 
		
\(\limsup\limits_{n\rightarrow\infty}\frac{1}{n}\sum\limits_{j=0}^{n-1}d(f_{w}^{j}(z), x_{j})<\epsilon\) (resp. \(\limsup\limits_{n\rightarrow\infty}\frac{1}{n}\sum\limits_{j=0}^{n-1}d(f^{j}(z), x_{j})<\epsilon\)). 
		
\item [(3)] \(\Gamma\) (resp. $f$) is said to have the \(w\)-asymptotic average shadowing property (resp. asymptotic average shadowing property) if every \(w\)-asymptotic average pseudo orbit (resp. asymptotic average pseudo orbit) $(x_j)_{j\in\mathbb{N}}$ is \(w\)-asymptotically shadowed on average (resp. asymptotically shadowed on average) by a point $z\in X$, i.e.,
		
\(\lim\limits_{n\rightarrow\infty}\frac{1}{n}\sum\limits_{j=0}^{n-1}d(f_{w}^{j}(z), x_{j})=0\) (resp. \(\lim\limits_{n\rightarrow\infty}\frac{1}{n}\sum\limits_{j=0}^{n-1}d(f^{j}(z), x_{j})=0\)). 
		
\item [(4)] \(\Gamma\) (resp. $f$) is said to have the \(w\)-weak asymptotic average shadowing property (resp. weak asymptotic average shadowing property) if for every $\epsilon>0$, any \(w\)-asymptotic average pseudo orbit (resp. asymptotic average pseudo orbit) $(x_j)_{j\in\mathbb{N}}$ is $(\epsilon, w)$-shadowed on average (resp. $\epsilon$-shadowed on average) by a point $z\in X$, i.e.,
		
\(\limsup\limits_{n\rightarrow\infty}\frac{1}{n}\sum\limits_{j=0}^{n-1}d(f_{w}^{j}(z), x_{j})<\epsilon\) (resp. \(\limsup\limits_{n\rightarrow\infty}\frac{1}{n}\sum\limits_{j=0}^{n-1}d(f^{j}(z), x_{j})<\epsilon\)).
		
\item [(5)] \(\Gamma\) (resp. $f$) is said to have the $w$-almost asymptotic average shadowing property (resp. almost asymptotic average shadowing property) if for every $\epsilon>0$, there is $\delta>0$ such that every \((\delta,w)\)-weak asymptotic average pseudo orbit (resp. $\delta$-weak asymptotic average pseudo orbit) $(x_j)_{j\in\mathbb{N}}$ is $(\epsilon, w)$-shadowed on average (resp. $\epsilon$-shadowed on average) by a point $z\in X$, i.e.,
		
\(\limsup\limits_{n\rightarrow\infty}\frac{1}{n}\sum\limits_{j=0}^{n-1}d(f_{w}^{j}(z), x_{j})<\epsilon\) (resp. \(\limsup\limits_{n\rightarrow\infty}\frac{1}{n}\sum\limits_{j=0}^{n-1}d(f^{j}(z), x_{j})<\epsilon\)).
		
\item [(6)] $\Gamma$ (resp. $f$) is said to have the $w$-$M_{\alpha}$-shadowing property (resp. $M_{\alpha}$-shadowing property) if for every $\epsilon>0$, there is $\delta>0$ such that every ($\delta$,$w$)-ergodic pseudo orbit (resp. $\delta$-ergodic pseudo orbit) $\xi=(x_j)_{j\in\mathbb{N}}$ is $M_{\alpha}$-($\epsilon$,$w$)-shadowed (resp. $M_{\alpha}$-$\epsilon$-shadowed) by a point $z\in X$, i.e., $B(z,\xi,w,\epsilon)\in M_{\alpha}$ (resp. $B(z,\xi,\epsilon)\in M_{\alpha}$), where $B(z,\xi,w,\epsilon)=\lbrace j\in\mathbb{N}$ $\mid$ $d(f_w^j(z),x_{j})<\epsilon\rbrace$ (resp. $B(z,\xi,\epsilon)=\lbrace j\in\mathbb{N}$ $\mid$ $d(f^j(z),x_{j})<\epsilon\rbrace$). 
		
	\end{enumerate}
\end{definition} 

\section{Variants of Shadowing and Their Equivalence} \label{main}    

The main result of this paper in the form of the Theorem \ref{MT} proves that the $w$-average shadowing property, the $w$-weak asymptotic average shadowing property, the $w$-mean ergodic shadowing property, the $w$-almost asymptotic average shadowing property, the $w$-$M_{\alpha}$-shadowing property for every $\alpha\in (0,1)$ and the $w$-asymptotic average shadowing property are equivalent for a finitely generated free semigroup action.
\medskip

We first discuss two lemmas and one remark which will be used in proving the main theorem.          

\begin{lemma} 
Let \(\Gamma=\Gamma(f_0,f_1, \ldots, f_m)\) and \(w=(w_j)_{j\in\mathbb{N}}\in\Sigma^m\). Then, \(\Gamma\) has the $w$-mean ergodic shadowing property if and only if for every $\epsilon>0$, there is $\delta>0$ such that every ($\delta$,$w$)-ergodic pseudo orbit is ($\epsilon$,$w$)-shadowed except on a set of upper density $\epsilon$.        
\label{2.1.1} 
\end{lemma} 

\begin{proof}
Suppose that $\Gamma$ has the $w$-mean ergodic shadowing property. For every $\epsilon>0$, there is $\delta>0$ such that if $\xi=(x_j)_{j\in\mathbb{N}}$ is a $\delta$-ergodic pseudo orbit, then there is $z\in X$ such that

$\limsup\limits_{n\to\infty}\frac{1}{n}\sum\limits_{j=0}^{n-1}d(f^j_{w}(z),x_j)<\epsilon^2.$ 
	
If $E=\lbrace j\in\mathbb{N}\mid d(f^j_{w}(z),x_j)\geq \epsilon\rbrace$, then we have

$$\epsilon^2>\limsup\limits_{n\to\infty}\frac{1}{n}\sum\limits_{j=0}^{n-1}d(f^j_{w}(z),x_j)\geq \limsup\limits_{n\to\infty}\frac{\epsilon}{n}(|E\cap \lbrace 0,1,\ldots,n-1\rbrace|)=\epsilon\overline{d}(E).$$
This implies that $\overline{d}(E)<\epsilon$. 
\medskip
	
Conversely, fix $\epsilon>0$ and choose $\eta<\frac{\epsilon}{\diam(X)+1}$. Let $\delta>0$ be such that if $\xi=(x_j)_{j\in\mathbb{N}}$ is a $\delta$-ergodic pseudo orbit, then $\xi$ is $\eta$-shadowed except on a set of upper density less than $\epsilon$, by some point $z$ in $X$. 
	
If $E=\lbrace j\in\mathbb{N}\mid d(f^j_{w}(z),x_j)\geq \eta\rbrace$, then we have $\overline{d}(E)<\eta$ and

$$\limsup\limits_{n\to\infty}\frac{1}{n}\sum_{j=0}^{n-1} d(f^j_{w}(z),x_j)
\leq \limsup\limits_{n\to\infty}\frac{1}{n}(\diam(X)|E\cap \lbrace 0,1,\ldots,n-1\rbrace|+\eta n)$$
$$\leq \diam(X)\overline{d}(E)+\eta<\epsilon.$$
	
This implies that $\Gamma$ has the $w$-mean ergodic shadowing property.  
\end{proof}

In reference to the above lemma, the following definition of the $w$-mean ergodic shadowing property is equivalent to the earlier definition and is useful for our study.

\begin{definition}
Let \(\Gamma=\Gamma(f_0,f_1, \ldots, f_m)\) and \(w=(w_j)_{j\in\mathbb{N}}\in\Sigma^m\). Then, \(\Gamma\) satisfies the $w$-mean ergodic shadowing property if for every \(\epsilon>0\), there is \(\delta>0\) such that if \(\xi=(x_j)_{j\in\mathbb{N}}\) is a \((\delta,w)\)-ergodic pseudo orbit, then there is \(z\) in \(X\) such that \(\overline{d}(B^c(z,\xi,w,\epsilon))<\epsilon\), where \(B^c(z,\xi,w,\epsilon)=\{j\in\mathbb{N}\;|\;d(f^j_{w}(z),x_j)\geq \epsilon\}.\)  
\end{definition} 

\begin{lemma}
	\label{epoapo} 
	Let \(\Gamma=\Gamma(f_0,f_1, \ldots, f_m)\) and \(w=(w_j)_{j\in\mathbb{N}}\in\Sigma^m\). Then, for any $\delta>0$ and any $(\frac{\delta}{2},w)$-ergodic pseudo orbit $\xi=\lbrace x_i\rbrace_{i\in\mathbb{N}}$, there exists a $(\delta,w)$-average pseudo orbit $\lbrace y_i\rbrace_{i\in\mathbb{N}}$ such that the set $\lbrace i\in\mathbb{N}$ $\mid$ $x_i\neq y_i\rbrace$ has density zero.   
\end{lemma} 

\begin{proof}
	Since $X$ is a compact metric space, $\diam(X)$ is finite and therefore, one can choose a natural number $M$ (sufficiently large) such that $\frac{\diam(X)}{M}<\frac{\delta}{8}$. If $d(f_{w_i}(x_i),x_{i+1})<\frac{\delta}{2}$ for all but finitely many $i\in\mathbb{N}$, then $\lbrace x_i\rbrace_{i\in\mathbb{N}}$ itself is the desired $(\delta,w)$-average pseudo orbit.   
	\medskip
	
	Let us now assume that $( n_i)_{i\in\mathbb{N}}$ is a strictly increasing sequence of natural numbers such that $B^c(\xi,w,\frac{\delta}{2})=\lbrace i\in\mathbb{N}$ $\mid$ $d(f_{w_i}(x_i),x_{i+1})\geq \frac{\delta}{2}\rbrace=\lbrace n_i$ $\mid$ $i\in\mathbb{N}\rbrace$. Let us define $( k_i)_{i\in\mathbb{N}}$ as follows: 
		
	$k_0=n_0$ and for $n\in\mathbb{N}$, $k_{n+1}=$ min $\lbrace l\in B^c(\xi,w,\frac{\delta}{2})$ $\mid$ $l\geq k_n+M\rbrace$.  
	
	If $J=\lbrace k_n$ $\mid$ $n\in\mathbb{N}\rbrace$, then $\mathcal{J}=\bigcup\limits_{i=0}^{M-1} (J+i)$ satisfies the following conditions:
	(a) $\mathcal{J}$, $\mathcal{J}+1$, $\mathcal{J}+2$ $\ldots$, $\mathcal{J}+M-1$ are disjoint, (b) $B^c(\xi,w,\frac{\delta}{2})\subset \mathcal{J}$, and (c) $\mathcal{J}$ has density zero.   
	\medskip
	
	Let us define $\lbrace y_i\rbrace_{i\in\mathbb{N}}$ as follows: 
	
	$$y_i=\begin{cases}
		f_{w}^{i-k_n}(x_{k_n}) & \mbox{if }  i\in [k_n,k_n+M) \mbox{ for some } n\in\mathbb{N} \\
		x_i & \mbox{otherwise.}   
	\end{cases}$$ 
	
	If for $k\geq 0$ and $n>0$, we have $A_k^n=\lbrace i\in [k,k+n)$ $\mid$ $d(f_{w_i}(y_i),y_{i+1})\geq \frac{\delta}{2}\rbrace$, then observe that $|A_k^n|\leq \frac{2(n+M)}{M}$. If we prove that $\lbrace y_i\rbrace_{i\in\mathbb{N}}$ is a $(\delta,w)$-average pseudo orbit, then we are through because $\lbrace i\in\mathbb{N}$ $\mid$ $x_i\neq y_i\rbrace\subset \mathcal{J}$.   
	
	Fix $n\geq M$ and $k\geq 0$. 
	
	If $[k,k+n)\cap \mathcal{J}=\phi$, then it is clear that $\frac{1}{n} \sum\limits_{i=0}^{n-1} d(f_{w_i}(y_{i+k}),y_{i+k+1})<\frac{\delta}{2}<\delta$. 
	
	If $[k,k+n)\cap \mathcal{J}\neq \phi$, then 
	
	$|\lbrace i\in [k,k+n)$ $\mid$ $d(f_{w_i}(y_i),y_{i+1})\geq \frac{\delta}{2}\rbrace|=|A_k^n|\leq \frac{2(n+M)}{M}\leq \frac{4n}{M}.$ 
    
    Hence, 
	$\frac{1}{n}\sum\limits_{i=0}^{n-1} d(f_{w_i}(y_{i+k}),y_{i+k+1})=\frac{1}{n} \sum\limits_{i\in A_k^n} d(f_{w_i}(y_i),y_{i+1})+\frac{1}{n}\sum\limits_{i\in [k,k+n)\setminus A_k^n} d(f_{w_i}(y_i),y_{i+1})$\\
	$\leq \frac{|A_k^n|}{n} \diam(X)+\frac{\delta}{2}\leq \frac{4}{M}\diam(X)+\frac{\delta}{2}<\delta$.  
\end{proof} 

\begin{remark} \label{AE}   
	\cite[Theorem 1.20]{W} If $(a_n)_{n\in\mathbb{N}}$ is a bounded sequence of non-negative real numbers then the following are equivalent:
	\begin{enumerate}
		\item [(a)] $\lim\limits_{n\to\infty} \frac{1}{n} \sum\limits_{i=0}^{n-1}  a_{i}=0$.
		\item [(b)] There exists $J\subset \mathbb{N}$ with $d(J)=0$ such that $\lim\limits_{n\rightarrow \infty}a_{n}=0$, provided $n\notin J$. 
	\end{enumerate}
\end{remark}  

With the above setup, we are now in a stage to prove the main result of this paper.  
\\

\noindent\textbf{Proof of Theorem~\ref{MT}:} 


\medskip

We first obtain that (b), (c) and (d) are equivalent by proving that (b)$\Rightarrow$(d), (d)$\Rightarrow$(c) and (c)$\Rightarrow$(b).   
\medskip

(b)$\Rightarrow$(d) Suppose that $\Gamma$ does not have the $w$-almost asymptotic average shadowing property. Then there exists $\epsilon>0$ such that for any $k\in \mathbb{N}$,
there exists a $(1/k,w)$-weak asymptotic average pseudo orbit
\(\beta^{(k)}=(\beta^{(k)}_{i})_{i\in\mathbb{N}}\) such that
\begin{equation}\label{NotShadowed}
	\limsup_{n\rightarrow\infty}
	\frac{1}{n}\sum_{i=0}^{n-1}d(f^{i}_{w}(z), \beta^{(k)}_{i})\geq \epsilon,\ \ \forall z\in X
\end{equation}

If we can construct a $w$-asymptotic average pseudo orbit
which is not $(\frac{\epsilon}{2},w)$-shadowed on average by any
point in $X$, then we are through.  
\medskip

Since $\beta^{(k)}$ is a
($1/k$,$w$)-weak asymptotic average pseudo orbit, there exists $N_{k}\in \mathbb{N}$
such that for all $n\geq N_{k}$, we have $\frac{1}{n}\sum\limits_{j=0}^{n-1}d(f_{w_j}(\beta_{j}^{(k)}), \beta_{j+1}^{(k)})<\frac{1}{k}.$  Without loss of generality, assume that $\{N_{k}\}_{k\in\mathbb{N}}$ 
is a strictly increasing sequence.  

We now set $m_{1}=2^{N_{2}}$ and define $\lbrace m_i\rbrace_{i\in\mathbb{Z}^+}$ using the method of induction. Suppose that we have already
defined $m_n$ for some $n\geq 1$. By \eqref{NotShadowed}, for
any $z\in X$ and any $k,N>0$ there exist $L>N$
and $\gamma>0$ such that if $d(z,y)<\gamma$,
$$\frac{1}{L+1}\sum_{i=0}^{L}d(f_{w}^{i}(y), \beta^{(k)}_{i})\geq \frac{\epsilon}{2}.$$

By the compactness of $X$, there are
$k_{n+1}\in\mathbb{N}$ and positive integers $L^{(n+1)}_{1}, \ldots, L^{(n+1)}_{k_{n+1}}\geq 2^{(n+1)m_n} 
$ such that for any $z\in X$ there is
$1\leq i \leq k_{n+1}$ satisfying
\begin{equation}\label{Exists_z}
	\frac{1}{L^{(n+1)}_{i}+1}\sum_{j=0}^{L^{(n+1)}_{i}}d(f_{w}^{j}(z),
	\beta_{j}^{(n+1)})\geq \frac{\epsilon}{2}.
\end{equation}

We denote $m_{n+1}=\max\left\{2^{N_{n+2}}, L^{(n+1)}_{1},
L^{(n+1)}_{2}, \ldots, L^{(n+1)}_{k_{n+1}}\right\}
$
and consider 
$
\xi=\left\{x_i\right\}_{i=0}^{\infty}=\beta_{0}^{(1)}
\beta_{1}^{(1)}\cdots\beta_{m_{1}}^{(1)}
\beta_{0}^{(2)}\beta_{1}^{(2)}\cdots
\beta_{m_{2}}^{(2)}\cdots\beta_{0}^{(n)}
\beta_{1}^{(n)}\cdots\beta_{m_{n}}^{(n)} \cdots 
$

We first show that $\xi$ is a $w$-asymptotic average pseudo orbit.
Denote $M_{0}=0$ and for $n\geq 1$, set $M_{n}=\sum\limits_{i=1}^n (m_i+1)$. One may easily verify that $M_n\leq n(m_n+1)\leq (n+1)m_n$ and hence $M_{n+1}\geq
m_{n+1} \geq 2^{M_n}$. Similarly, $M_n\geq 2^{N_{n+1}}> N_{n+1}$.
\medskip

Fix $n\geq 1$ and an integer $j\in [M_n,M_{n+1})$.
By the definition of $\xi$ we obtain
\begin{align*}
&\frac{1}{j}\sum_{i=0}^{j-1} d(f_{w_i}(x_{i}), x_{i+1})\\
&= \frac{1}{j}\left[
	\sum_{k=1}^{n}\sum_{i=M_{k-1}}^{M_{k}-2}d(f_{w_i}(x_{i}), x_{i+1})+\sum_{i=1}
	^{n}d(f_{w_{M_{i}-1}}(x_{M_{i}-1}), x_{M_{i}})+
	\sum_{i=M_{n}}^{j-1}d(f_{w_i}(x_{i}), x_{i+1})\right]. 
\end{align*}

Since $\frac{1}{m_{k}}\sum\limits_{i=M_{k-1}}^{M_{k}-2}d(f_{w_i}(x_{i}), x_{i+1})<\frac{1}{k}$, we obtain $$\frac{1}{j}
\sum\limits_{k=1}^{n}\sum\limits_{i=M_{k-1}}^{M_{k}-2}d(f_{w_i}(x_{i}), x_{i+1})
+\frac{1}{j}\sum\limits_{i=1}
^{n}d(f_{w_{M_{i}-1}}(x_{M_{i}-1}), x_{M_{i}})
\leq \sum\limits_{k=1}^{n}\frac{m_{k}}{j k}+\frac{n \diam(X)}{2^n}$$
$$\leq \frac{1}{n}+\sum\limits_{k=1}^{n-1}\frac{m_{k}}{j k}+\frac{n \diam(X)}{2^n}\leq \frac{1}{n}+\frac{M_{n-1}}{M_n}+\frac{n \diam(X)}{2^n}\leq \frac{1}{n}+\frac{M_{n-1}}{2^{M_{n-1}}}+\frac{n \diam(X)}{2^n}.$$ 

Additionally, observe that if $j\leq M_n + N_{n+1}$, then
$$
\frac{1}{j}\sum_{i=M_{n}}^{j-1}d(f_{w_i}(x_{i}), x_{i+1})\leq
\frac{N_{n+1}}{j}\diam(X)\leq \frac{N_{n+1}}{2^{N_{n+1}}}\diam(X),
$$

and in the second case of $j> M_n + N_{n+1}$, by the choice of
$N_{n+1}$, we immediately obtain
$$\frac{1}{j}\sum\limits_{i=M_{n}}^{j-1}d(f_{w_i}(x_{i}), x_{i+1})\leq
\frac{1}{j}\sum\limits_{i=M_{n}}^{j-1}d(f_{w_{i-M_n}}(\beta^{(n+1)}_{i-M_n}), \beta^{(n+1)}_{i-M_n+1})
\leq \frac{1}{n+1}.
$$ 
Therefore, $\lim\limits_{j\to \infty} \frac{1}{j}
\sum\limits_{i=0}^{j-1} d(f_{w_i}(x_{i}), x_{i+1})=0$, so indeed $\xi$ is a
$w$-asymptotic average pseudo orbit and the claim holds.  
\medskip

Fix $z\in X$. By \eqref{Exists_z}, for any $n\in \mathbb{N}$ and
point $f_{w}^{M_n}(z)$ we can select $1\leq i_{n} \leq k_{n+1}$ such
that
$$\frac{1}{L^{(n+1)}_{i_{n}}+1}
\sum\limits_{j=M_{n}}^{M_{n}+L_{i_n}^{(n+1)}} d(f_{w}^{j}(z), x_{j})=\frac{1}{L^{(n+1)+1}_{i_n}}
\sum\limits_{j=0}^{L^{(n+1)}_{i_n}}d(f_{w}^{j}(f^{M_{n}}(z)),
\beta_{j}^{(n+1)})\geq \frac{\epsilon}{2}.$$ 

Therefore, 
\begin{align*}
	&\limsup_{n\rightarrow\infty}\frac{1}{n}
	\sum_{j=0}^{n-1}d(f_{w}^{j}(z),x_{j})\ \geq\ \limsup_{n\rightarrow\infty}
	\frac{1}{M_{n}+L_{i_n}^{(n+1)}+1}
	\sum_{j=0}^{M_{n}+L_{i_n}^{(n+1)}}d(f_{w}^{j}(z), x_{j})\\
	&\qquad\qquad\quad\ \geq\ \limsup_{n\rightarrow\infty}
	\frac{L^{(n+1)}_{i_n}+1}{M_{n}+L_{i_n}^{(n+1)}+1}\frac{1}{L^{(n+1)}_{i_n}+1}
	\sum_{j=M_{n}}^{M_{n}+L_{i_n}^{(n+1)}}d(f_{w}^{j}(z), x_{j})\\
	&\qquad\qquad\quad\ \geq\ \frac{\epsilon}{2}\limsup_{n\rightarrow\infty}
	\frac{L^{(n+1)}_{i_n}+1}{M_{n}+L_{i_n}^{(n+1)}+1}
	\ \geq\ \frac{\epsilon}{2}\limsup_{n\rightarrow\infty}
	\frac{2^{M_{n}}}{M_{n}+2^{M_{n}}+1}= \frac{\epsilon}{2}.
\end{align*}

This means that $\xi$ is not $\frac{\epsilon}{2}$-shadowed on 
average by any point in $X$. Hence, $\Gamma$ does not have
the $w$-weak asymptotic average shadowing property.   
\bigskip 

(d)$\Rightarrow$(c) Suppose that $\Gamma$ has the $w$-almost asymptotic average shadowing property. Fix $\epsilon>0$ and let $\delta>0$ be given for $\epsilon$ by the $w$-almost asymptotic average shadowing property. Let $\lbrace x_i\rbrace_{i\in\mathbb{N}}$ be a $(\delta,w)$-ergodic pseudo orbit and let $D=\lbrace i\in\mathbb{N}$ $\mid$ $d(f_{w_i}(x_i),x_{i+1})\geq \delta\rbrace$. Then observe that $\overline{d}(D)=0$ and 

$\limsup\limits_{n\to\infty} \frac{1}{n}\sum\limits_{i=0}^{n-1} d(f_{w_i}(x_i),x_{i+1})$\\
$< \limsup\limits_{n\to\infty} \frac{1}{n} \left(\diam(X)|D\cap \lbrace 0,1,\ldots,n-1\rbrace|+{n\delta}\right)$\\
$= \diam(X)\overline{d}(D)+\delta=\delta.$  
\medskip

Thus, $\lbrace x_i\rbrace_{i\in\mathbb{N}}$ is a $(\delta,w)$-weak asymptotic average pseudo orbit and hence, by the $w$-almost asymptotic average shadowing property there is $z\in X$ such that $$\limsup\limits_{n\to\infty} \frac{1}{n}\sum\limits_{i=0}^{n-1} d(f_{w}^i(z),x_i)<\epsilon.$$ 
This shows that $\Gamma$ has the $w$-mean ergodic shadowing property.     
\bigskip

(c)$\Rightarrow$(b) Suppose that $\Gamma$ has the $w$-mean ergodic shadowing property. Fix $\epsilon>0$ and let $\delta>0$ be given by the $w$-mean ergodic shadowing property. Consider a $w$-asymptotic average pseudo orbit $\lbrace x_i\rbrace_{i\in\mathbb{N}}$, i.e., $\lim\limits_{n\to\infty} \frac{1}{n}\sum\limits_{i=0}^{\infty} d(f_{w_i}(x_i),x_{i+1})=0$. By the Remark \ref{AE}, there exists a subset $J\subset\mathbb{N}$ of density zero such that $\lim\limits_{n\to \infty} d(f_{w_i}(x_i),x_{i+1})=0$, provided $n\notin J$ which implies that $\lbrace x_i\rbrace_{i\in\mathbb{N}}$ is a $(\delta,w)$-ergodic pseudo orbit. By the $w$-mean ergodic shadowing property there exists $z\in X$ such that $\limsup\limits_{n\to\infty} \frac{1}{n}\sum\limits_{i=0}^{n-1} d(f_w^i(z),x_i)<\epsilon$. This shows that $\Gamma$ has the $w$-weak asymptotic average shadowing property.     
\medskip

We now obtain the equivalence of (c) and (e) by proving that (c)$\Rightarrow$(e) and (e)$\Rightarrow$(c). 
\medskip

(c)$\Rightarrow$(e) Suppose that $\Gamma$ has the $w$-mean ergodic shadowing property. Let $\alpha_0\in (0,1)$ be fixed and we want to show that $\Gamma$ satisfies the $w$-$M_{\alpha_0}$-shadowing property. Let $\epsilon_0\in (0,1)$ be such that $\epsilon_0=1-\alpha_0$, and let $\delta_0>0$ be given for this $\epsilon_0$ by the $w$-mean ergodic shadowing property of $\Gamma$. Then for any $(\delta_0,w)$-ergodic pseudo orbit $\xi=\lbrace x_i\rbrace_{i\in\mathbb{N}}$ there exists $x\in X$ such that $\overline{d}(B^c(x,\xi,w,\epsilon))<\epsilon_0$. Since $\underline{d}(A)=1-\overline{d}(A^c)$ for any $A\subset\mathbb{N}$, we have that $\underline{d}(B(x,\xi,w,\epsilon_0))>1-\epsilon_0=\alpha_0$. If $\epsilon>\epsilon_0$, then $\underline{d}(B(x,\xi,w,\epsilon))\geq\underline{d}(B(x,\xi,w,\epsilon_0))>\alpha_0$. If $\epsilon<\epsilon_0$, then by the $w$-mean ergodic shadowing property of $\Gamma$ there exists $\delta>0$ such that every $(\delta,w)$-ergodic pseudo orbit $\xi=\lbrace x_i\rbrace_{i\in\mathbb{N}}$ satisfies that $\overline{d}(B^c(x,\xi,w,\epsilon))<\epsilon$ for some $x\in X$. So, $\underline{d}(B(x,\xi,w,\epsilon))=1-\overline{d}(B^c(x,\xi,w,\epsilon))>1-\epsilon>1-\epsilon_0=\alpha_0$. Therefore, we conclude that $\Gamma$ has the $w$-$M_{\alpha}$-shadowing property for every $\alpha\in (0,1)$.                
\bigskip 

(e)$\Rightarrow$(c) Suppose that $\Gamma$ has the $w$-$M_{\alpha}$-shadowing property for every $\alpha\in (0,1)$. Let $\epsilon_0\in (0,1)$ be fixed and $\epsilon_0=1-\alpha_0$ for some fixed $\alpha_0\in (0,1)$. For this $\epsilon_0$ there exists $\delta_0>0$ by the $w$-$M_{\alpha_0}$-shadowing property of $\Gamma$. Let $\xi=\lbrace x_i\rbrace_{i\in\mathbb{N}}$ be a $(\delta_0,w)$-ergodic pseudo orbit and hence, there exists $x\in X$ such that $\underline{d}(B(x,\xi,w,\epsilon_0))>\alpha_0$. Since $\underline{d}(A)=1-\overline{d}(A^c)$ for any $A\subset\mathbb{N}$, therefore $\overline{d}(B^c(x,\xi,w,\epsilon_0))<1-\alpha_0=\epsilon_0$. If $\epsilon\geq 1$, then choose the same $\delta_0$ as was given for $\epsilon_0$ by the $w$-$M_{\alpha_0}$-shadowing property of $\Gamma$. Then any $(\delta_0,w)$-ergodic pseudo orbit $\xi=\lbrace x_i\rbrace_{i\in\mathbb{N}}$ satisfies that $\overline{d}(B^c(x,\xi,w,\epsilon))\leq\overline{d}(B^c(x,\xi,w,\epsilon_0))<1-\alpha_0=\epsilon_0<\epsilon$. Therefore, by the Lemma \ref{2.1.1}, we conclude that $\Gamma$ has the $w$-mean ergodic shadowing property.
\medskip  

Finally, we prove that (a)$\Rightarrow$(b), (c)$\Rightarrow$(f) and (f)$\Rightarrow$(a) to complete our proof of the theorem.  
\medskip

(a)$\Rightarrow$(b) Suppose that $\Gamma$ has the $w$-average shadowing property. Fix an $\epsilon>0$ and a $w$-asymptotic average pseudo orbit $\lbrace x_i\rbrace_{i\in\mathbb{N}}$. By the $w$-average shadowing property of $\Gamma$, there exists $\delta>0$ such that every $(\delta,w)$-average pseudo orbit is $(\epsilon,w)$-shadowed on average by some point in $X$. By the Remark \ref{AE}, observe that $\lbrace x_i\rbrace_{i\in\mathbb{N}}$ is a $(\frac{\delta}{2},w)$-ergodic pseudo orbit. By the Lemma \ref{epoapo}, there is a $(\delta,w)$-average pseudo orbit $\lbrace y_i\rbrace_{i\in\mathbb{N}}$ such that $E=\lbrace i\in\mathbb{N}$ $\mid$ $x_i\neq y_i\rbrace$ has density zero. So there exists $z\in X$ such that 
$\limsup\limits_{n\to \infty} \frac{1}{n}\sum\limits_{i=0}^{n-1} d(f_{w}^i(z),y_i)<\epsilon.$ 
Then, 
$$\limsup\limits_{n\to \infty} \frac{1}{n}\sum\limits_{i=0}^{n-1} d(f_{w}^i(z),x_i)\leq\limsup\limits_{n\to\infty} \frac{1}{n}\sum\limits_{i=0}^{n-1} \left(d(f_{w}^i(z),y_i)+d(x_i,y_i)\right)$$  

$$<\epsilon+\limsup\limits_{n\to\infty} \frac{1}{n} \sum\limits_{i=0}^{n-1} d(x_i,y_i)\\
\leq\epsilon+\diam(X)\overline{d}(E)=\epsilon.$$ This shows that $\Gamma$ has the $w$-weak asymptotic average shadowing property.  
\bigskip

(c)$\Rightarrow$(f) Suppose that $\Gamma$ has the $w$-mean ergodic shadowing property. We want to show that $\Gamma$ has the $w$-asymptotic average shadowing property. Let us fix $\epsilon>0$ and let $\lbrace x_i\rbrace_{i\in\mathbb{N}}$ be a $w$-asymptotic average pseudo orbit. Then by the Remark \ref{AE}, $\lbrace x_i\rbrace_{i\in\mathbb{N}}$ is a $(\delta,w)$-ergodic pseudo orbit for any $\delta>0$. Therefore, we can find a sequence $\lbrace z_m\rbrace_{m\in\mathbb{N}}$ in $X$ such that $$\limsup\limits_{n\to \infty} \frac{1}{n}\sum\limits_{i=0}^{n-1} d(f_{w}^i(z_m),x_i)<\frac{\epsilon}{2^m}.$$ 

Without loss of generality we can assume that $\lbrace z_m\rbrace_{m\in\mathbb{N}}$ converges to $z$ in $X$. We claim that $\limsup\limits_{n\to\infty}\frac{1}{n}\sum\limits_{i=0}^{n-1} d(f_{w}^i(z),x_i)=0$. 

Since $z_m$ converges to $z$ in $X$, there exists a subsequence $\lbrace z_{m_k}\rbrace_{k\in\mathbb{N}}$ of $z_m$ such that $d(f_{w}^i(z_{m_k}),f_{w}^i(z))<\frac{\epsilon}{2^k}$. Then by the triangle inequality, we get that $d(f_{w}^i(z),x_i)\leq d(f_{w}^i(z_{m_k}),f_{w}^i(z))+d(f_{w}^i(z_{m_k}),x_i)<\frac{\epsilon}{2^k}+\frac{\epsilon}{2^{m_k}}\leq \frac{\epsilon}{2^k}+\frac{\epsilon}{2^k}.$ So, $\frac{1}{n}\sum\limits_{i=0}^{n-1} d(f_{w}^i(z),x_i)\leq \frac{\epsilon}{2^{k-1}}$ and hence, $\limsup\limits_{n\to\infty} \frac{1}{n}\sum\limits_{i=0}^{n-1} d(f_{w}^i(z),x_i)\leq \frac{\epsilon}{2^{k-1}}$ which proves the claim and completes the proof.    
\bigskip

(f)$\Rightarrow$(a) This implication was proved in Theorem 1.2 of \cite{ZWFL}. 
\medskip

\begin{example}
Let \(D\) be the closed unit disk in \(\mathbb{R}^2\) and \(f_0,f_1,f_2:D\to D\) be given by $f_0(x,y)=(x,y)$, $f_1(x,y) = (y,x)$ and $f_2(x,y) = \left(\frac{x}{2},\frac{y}{2}\right)$, respectively. We prove that $\Gamma=\Gamma(f_0,f_1,f_2)$ has the $w$-asymptotic average shadowing property, where $w$ is the sequence $\lbrace w_0w_1w_2w_3\cdots\rbrace$ with $w_{2i}=1$ for $i\in\mathbb{N}$ and $w_{2i+1}=2$ for $i\in\mathbb{N}$.  
\medskip
	
Let $\lbrace (x_i,y_i)\rbrace_{i\in\mathbb{N}}$ be a $w$-asymptotic average pseudo orbit for $\Gamma$. For $i\in\mathbb{N}$, if we set $\alpha_i=d(f_{w_i}(x_i,y_i),(x_{i+1},y_{i+1}))$ then $\lim\limits_{n\to\infty} \frac{1}{n}\sum\limits_{i=0}^{n-1}\alpha_i=0$. Let $\lbrace (a_i,b_i)\rbrace_{i\in\mathbb{N}}$ be an orbit such that $(a_0,b_0)\in D$ and $(a_{i+1},b_{i+1})=f_{w_i}((a_i,b_i))$ for all $i\in\mathbb{N}$. If we show that $\lim\limits_{n\to\infty}\frac{1}{n}\sum\limits_{i=0}^{n-1}d((a_i,b_i),(x_i,y_i))=0$, then we are through. Let us set $M=d((a_0,b_0),(x_0,y_0))$. Then, 
$$d((a_1,b_1),(x_1,y_1))\leq d((x_1,y_1),f_{w_0}(x_0,y_0))+d(f_{w_0}(x_0,y_0),f_{w_0}(a_0,b_0))\leq \alpha_0+M.$$
$$d((a_2,b_2),(x_2,y_2))\leq d((x_2,y_2),f_{w_1}(x_1,y_1))+d(f_{w_1}(x_1,y_1),f_{w_1}(a_1,b_1))$$
$$\leq \alpha_1+d((\frac{x_1}{2},\frac{y_1}{2}),(\frac{a_1}{2},\frac{b_1}{2}))\leq \alpha_1+\frac{1}{2}d((x_1,y_1),(a_1,b_1))\leq \alpha_1+\frac{1}{2}(\alpha_0+M).$$ 
$$d((a_3,b_3),(x_3,y_3))\leq d((x_3,y_3),f_{w_2}(x_2,y_2))+d(f_{w_2}(x_2,y_2),f_{w_2}(a_2,b_2))$$
$$\leq \alpha_2+d((y_2,x_2),(b_2,a_2))\leq \alpha_2+d((x_2,y_2),(a_2,b_2)\leq \alpha_2+(\alpha_1+\frac{1}{2}\alpha_0+\frac{1}{2}M).$$ 
$$d((a_4,b_4),(x_4,y_4))\leq d((x_4,y_4),f_{w_3}(x_3,y_3))+d(f_{w_3}(x_3,y_3),f_{w_3}(a_3,b_3))$$
$$\leq \alpha_3+d((\frac{x_3}{2},\frac{y_3}{2}),(\frac{a_3}{2},\frac{b_3}{2}))\leq \alpha_3+\frac{1}{2}d((x_3,y_3),(a_3,b_3))$$  
$$\leq\alpha_3+\frac{1}{2}(\alpha_2+\alpha_1+\frac{1}{2}\alpha_0+\frac{1}{2}M).$$
	
A careful observation helps us to write the following recurrent inequalities:  
	
$d((a_k,b_k),(x_k,y_k))\leq \alpha_{k-1}+\alpha_{k-2}+\frac{1}{2}(\alpha_{k-3}+\alpha_{k-4})+\cdots+\frac{1}{2^{k-2}}(\alpha_0+M)$ for $k$ odd and $k\geq 5$.   
	
$d((a_k,b_k),(x_k,y_k))\leq \alpha_{k-1}+\frac{1}{2}(\alpha_{k-2}+\alpha_{k-3})+\frac{1}{2^2}(\alpha_{k-4}+\alpha_{k-5})+\cdots+\frac{1}{2^{k-2}}(\alpha_0+M))$ for even $k$ and $k\geq 6$.  
	
Then, $\sum\limits_{i=0}^{n-1}d((a_i,b_i),(x_i,y_i))\leq [M(1+1+\frac{1}{2}+\frac{1}{2}+\frac{1}{2^2}+\frac{1}{2^2}+\frac{1}{2^3}+\frac{1}{2^3}+\frac{1}{2^4}+\frac{1}{2^4}+\cdots+\frac{1}{2^{\frac{n-1}{2}}})]+[\alpha_0(1+1+\frac{1}{2}+\frac{1}{2}+\frac{1}{2^2}+\frac{1}{2^2}+\frac{1}{2^3}+\frac{1}{2^3}+\frac{1}{2^4}+\frac{1}{2^4}+\cdots+\frac{1}{2^{\frac{n-1}{2}}})]+[\alpha_1(1+1+\frac{1}{2}+\frac{1}{2}+\frac{1}{2^2}+\frac{1}{2^2}+\frac{1}{2^3}+\frac{1}{2^3}+\frac{1}{2^4}+\frac{1}{2^4}+\cdots+\frac{1}{2^{\frac{n-1}{2}}})]+\cdots+[\alpha_{n-1}(1+1+\frac{1}{2}+\frac{1}{2}+\frac{1}{2^2}+\frac{1}{2^2}+\frac{1}{2^3}+\frac{1}{2^3}+\frac{1}{2^4}+\frac{1}{2^4}+\cdots+\frac{1}{2^{\frac{n-1}{2}}})]\leq 4\left(M+\sum\limits_{i=0}^{n-1}\alpha_i\right).$   
	\medskip
	
Therefore, $\lim\limits_{n\to\infty} \frac{1}{n}\sum\limits_{i=0}^{n-1}d((a_i,b_i),(x_i,y_i))\leq \lim\limits_{n\to \infty} \frac{1}{n}\left(4\left(M+\sum\limits_{i=0}^{n-1}\alpha_i\right)\right)=0$,  
\medskip
	
and this proves that $\Gamma$ has the $w$-asymptotic average shadowing property and hence, we conclude that $\Gamma$ satisfies any of the variants of the shadowing property listed in the Theorem \ref{MT} above.       
\end{example}

\end{document}